\documentclass[12pt]{article}
\usepackage[utf8]{inputenc}
\usepackage[english]{babel}
 
\usepackage[square,sort,comma,numbers]{natbib}
\usepackage{soul}

\usepackage[margin=3in]{geometry}

\usepackage[dvipsnames]{xcolor}
\usepackage{float}

\usepackage{graphicx}
\usepackage{tikz}
\usetikzlibrary{arrows}
\usetikzlibrary{matrix,arrows,decorations.pathmorphing}

\usepackage{amsmath}
\usepackage{amssymb}
\usepackage{amsthm}

\usepackage{stmaryrd}
\usepackage[all]{xy}
\usepackage[mathcal]{eucal}
\usepackage{verbatim}\usepackage{fullpage} \usepackage{wrapfig}
\usepackage{lipsum}
\usepackage[all]{xy}
\theoremstyle{plain}
\newtheorem{thm}{Theorem}[section]

\newtheorem{lem}[thm]{Lemma}

\newtheorem{prop}[thm]{Proposition}
\newtheorem{cor}[thm]{Corollary}

\newtheorem{fact}[thm]{Fact}

\newtheorem{qu}[thm]{Problem}
\theoremstyle{definition}

\theoremstyle{remark}

\newcommand{\nc}{\newcommand}
\nc{\dmo}{\DeclareMathOperator}
\DeclareMathOperator{\PConf}{PConf}

\DeclareMathOperator{\Conf}{Conf}
\DeclareMathOperator{\PB}{PB}
\DeclareMathOperator{\Mod}{Mod}

\DeclareMathOperator{\Diff}{Diff}

\DeclareMathOperator{\QQ}{\mathbb{Q}}

\DeclareMathOperator{\CC}{\mathbb{C}}

\DeclareMathOperator{\ZZ}{\mathbb{Z}}

\nc{\para}[1]{\medskip\noindent\textbf{#1.}}
\title{From pure braid groups to hyperbolic groups}
\author{Lei Chen}
\begin{document}
 \bibliographystyle{alpha}
\maketitle

\begin{abstract}
In this note we show that any homomorphism from a pure surface braid group to a torsion-free hyperbolic group  either has a cyclic image or  factors through a forgetful map. This extends and gives a new proof of an earlier result \cite{Chen} of the author which works only when the target is a free group or a surface group. We also prove a similar rigidity result for the pure braid group of the disk.

\end{abstract}
\section{Introduction}
Let $S_{g,p}$ be a closed surface of genus $g$ with $p$ punctures.
Let $\PConf_n(S_{g,p}):=S_{g,p}\times ...\times S_{g,p}-\triangle$ be the space of ordered $n$-tuples of distinct points in $S_{g,p}$ and let $PB_n(S_{g,p}):=\pi_1(\PConf_n(S_{g,p}))$ be the surface braid group. We call a group $\Lambda$ \emph{hyperbolic} if the Cayley graph of $\Lambda$ is $\delta$-hyperbolic for some $\delta>0$. Examples of hyperbolic groups include free groups and $\pi_1(M)$ when $M$ is a compact hyperbolic manifold. Surface braid groups have many homomorphisms to hyperbolic groups. Composing the map $p_i: PB_n(S_{g,p})\to \pi_1(S_{g,p})$ and $\Phi: \pi_1(S_{g,p})\to \Lambda$, where $p_i$ is the induced map on the fundamental groups of the natural projection $\PConf_n(S_{g,p})\to S_{g,p}$ to the $i$th coordinate and $\Lambda$ a hyperbolic group provides such an example. We can even obtain surjective homomorphisms from $PB_n(S_{g,p})$ to any finite generated hyperbolic groups as $g$ varies because $\pi_1(S_{g,p})$ has surjective homomorphisms to any finite generated group as $g$ varies. In this paper, we will classify all homomorphisms from $PB_n(S_{g,p})$ to torsion-free hyperbolic groups.

The result of this paper has a precursor. In  \cite{Chen} we proved that any surjective homomorphism $PB_n(S_g)\to \Lambda$, where $\Lambda$ is a nonabelian surface group or a nonabelian free group, factors through the natural projection $p_i$ for some $i$. The following theorem generalizes \cite{Chen} in two ways. First, the target is extended to all torsion-free, non-elementary hyperbolic groups; and second, the domain is extended to finite index subgroups of $PB_n(S_g)$ whose quotient is an abelian group. For simplicity, we omit $p$ whenever $p=0$. Let $\Gamma\lhd PB_n(S_{g,p})$ be a  finite index normal subgroup and let $\Gamma_i:=p_i(\Gamma)$.
\begin{thm}[Classification of homomorphisms for braid groups of closed surfaces]
\label{hyperbolic}
Let $n>0$ and $g>1$,  and let $\Lambda$ be a torsion-free, non-elementary hyperbolic group. If $\PB_n(S_g)/\Gamma$ is abelian, then any homomorphism $\Gamma\to \Lambda$ either factors through $p_i$ or its image is a cyclic group. 
\end{thm}
Understanding homomorphisms with cyclic image $\Gamma\to \ZZ$ (as we assume that the target group is torsion-free) is the same as computing $H^1(\Gamma;\mathbb{Q})$. The following theorem gives us a computation of the first betti number of $\Gamma$.
\begin{thm}[First betti number of braid groups of closed surfaces]
\label{bettinumber}
Let $n>0$ and $g>1$. If $PB_n(S_g)/\Gamma$ is abelian, then 
\[H^1(\Gamma;\QQ)=\bigoplus p_i^*(H^1(\Gamma_i;\QQ)).
\]
\end{thm}
For the  pure braid group of a punctured surface, we obtain a version of Theorem \ref{hyperbolic}  for the whole group $PB_n(S_{g,p})$ instead of its finite index subgroups. We do not know if statements in Theorem \ref{hyperbolic} and \ref{bettinumber} are valid or not for $PB_n(S_{g,p})$ when $p>0$.

\begin{thm}[Classification of homomorphisms for braid groups of punctured surfaces]\label{puncturehyperbolic}
Let  $n>0$ and $g>1$, and let $\Lambda$  be a torsion-free, non-elementary hyperbolic group.  Any homomorphism $PB_n(S_{g,p})\to \Lambda$  either factors through $p_i$ or its image is a cyclic group.
\end{thm}

We now discuss the same problem for the pure braid group $PB_n:=\pi_1(\PConf_n(\CC))$, where $\text{PConf}_n(\mathbb{C}):=\mathbb{C}\times...\times \mathbb{C}-\triangle$ is the space of ordered $n$-tuples of distinct points in $\mathbb{C}$. The corresponding statement in Theorem \ref{puncturehyperbolic} is no longer  true for $PB_n$.

It is well-known that the center of $PB_n$ is a cyclic group which is generated by the Dehn twist $Z_n$ about the boundary curve.  The quotient group $PB_3/Z_3$ is the free group of rank two $F_2$, which is hyperbolic. Thus the quotient homomorphism \[
Q: PB_3\to PB_3/Z_3\cong F_2\] is a surjective homomorphism to a hyperbolic group that does not factor through forgetful maps.

Moreover, there is a natural surjective homomorphism between braid groups that arises from a classical
miracle: ``resolving the quartic”. Indeed, let 
\[R: \PConf_4(\mathbb{C})\to \PConf_3(\mathbb{C})\]
 be the map given by 
 \[R(a, b, c, d) = (ab + cd, ac + bd, ad + bc)\]
 The induced homomorphism on fundamental groups $R_*:PB_4 \to  PB_3$  is a surjective homomorphism. Thus, we obtain another natural homomorphism 
 \[
 RQ:= Q \circ R_*: PB_4\to PB_3\to PB_3/Z_3\cong F_2.\] In this paper, we will prove that $RQ$ and $Q$ are  the  only exceptional homomorphisms. We call a map a \emph{forgetful map} if it is induced by forgetting points in $\PConf_n(\CC)$. For example, the map $f:\PConf_n(\CC)\to\PConf_3(\CC)$ defined by $f(x_1,...,x_n)=(x_1,x_2,x_3)$ is a forgetful map.
 \begin{thm}[Homomorphism classification for braid groups of the disk]
 \label{Pn}
Let $n\ge 3$,$\Lambda$ be a torsion-free, non-elementary hyperbolic group and $\phi: PB_n\to \Lambda$ be a homomorphism. Then either $\phi(PB_n)$ is a cyclic group or there exists $\phi': F_2\to \Lambda$ and a forgetful map $f_3: PB_n\to PB_3$ (resp. $f_4: PB_n\to PB_4$) such that $\phi=\phi'\circ Q\circ f_3$ (resp. $\phi=\phi'\circ RQ\circ f_4$).
 \end{thm}

Theorem \ref{Pn} is a generalization of the result \cite[Theorem 3.5]{DCohen}, which is a special case when $\Lambda$  is a free group. We remark that the statement in Theorem \ref{Pn} is not true when $\Lambda$ is a relative hyperbolic group.  Let $B_n(S^2)$ be the braid group of the two sphere $S^2$, which is the fundamental group of the space of unordered $n$-tuples of distinct points in $S^2$.  Deligne--Mostow \cite{DeligneMostow} constructed a homomorphism $E: B_n(S^2)\to SU(k,1)$ such that the image is a lattice. We also know that  $B_n(S^2)$ contains $PB_{n-2}$ as a subgroup and the restriction $E$ to $PB_{n-2}$ does not factor through a forgetful map.  Our proof does not apply when $\Lambda$ is a relative hyperbolic group because we strongly use a property of hyperbolic groups that the centralizer of any nontrivial element is cyclic. This property does not hold for relative hyperbolic groups.
 \vskip 0.3cm
 
At last, we ask the following natural question.
\begin{qu}
Are theorems \ref{hyperbolic}, \ref{bettinumber} and \ref{Pn} true  for all finite index subgroups of $PB_n(S_{g,p})$ and $PB_n$?
\end{qu}

\para{Comparing the methods in \cite{Chen}, \cite{DCohen} and the current paper}
The results in  \cite{Chen} and \cite{DCohen} only work for the free group $F_m$ since they use the property that the $H^1(F_m;\ZZ)$ is an isotropic subspace for the cup product. Since hyperbolic groups can be perfect, it seems impossible that this idea can be used to prove the results in this paper. Moreover, in \cite{Chen}, we use the method of F.E.A. Johnson \cite{FEAJohnson} and Salter \cite{Salter}, which strongly uses a special property of free groups and surface groups:  finitely-generated normal subgroups of either  free groups or surface groups are either trivial or have finite index. This property is certainly not true for general hyperbolic groups. For example, let $M$  be a 3-dimensional hyperbolic manifold that is  a surface bundle over the circle. Then $\pi_1(M)$ contains a surface subgroup as a nontrivial finitely-generated, infinite index normal subgroup. The novelty of this paper is the observation that the rigidity results such as Theorem \ref{hyperbolic} and Theorem \ref{Pn} are not consequences of the classification of isotropic subspaces of the first homology, but rather the rich commuting and lantern relations of the subgroups of the pure braid groups.

\para{Acknowledgment} The author would like  to thank Vlad Markovic for very useful discussions and Oishee Banerjee, Benson Farb, Dan Margalit and Nick Salter for comments on the draft.

\section{Obstructing homomorphisms to hyperbolic groups}
In this section, we discuss tools to obstruct homomorphisms to hyperbolic groups. We discuss the rigidity of homomorphisms to hyperbolic groups from two classes of groups: the $\ZZ$-central extensions and the direct product of groups.
\subsection{The Euler class of a $\ZZ$-central extensions}
For a $\ZZ$-central extension
\begin{equation}\label{ce}
1\to \ZZ\to G\xrightarrow{p} \overline{G}\to 1,
\end{equation}
we can associate an Euler class $Eu(p)\in H^2(\overline{G};\mathbb{Z})$ (see e.g., \cite[Chapter 4]{Brown}). We know that the exact sequence \eqref{ce} splits if and only if $Eu(p)=0\in H^2(G;\mathbb{Z})$. On the other hand, $Eu(p)\neq 0$ if and only if a nontrivial element of the $\mathbb{Z}$-subgroup of $G$ is a commutator in $G$.

We need the following standard fact about torsion-free, non-elementary hyperbolic groups.
\begin{fact}
If $\Lambda$ is a torsion-free hyperbolic group, then the centralizer of $1\neq h\in \Lambda$ is a cyclic group. If $\Lambda$ is not torsion-free and $h\in  \Lambda$ has infinite order, then the centralizer of $\Lambda$ is virtually cyclic.
\end{fact}
We call two elements $a$ and $b$ in a hyperbolic group \emph{independent} if their hyperbolic axes are different. If $a$ and $b$ are independent, then the intersection of their centralizers is the identity subgroup. The following theorem describes the rigidity of homomorphisms from $G$ to hyperbolic groups.
\begin{lem}\label{centralextension}
Let $\Lambda$ be  $\ZZ$ or a torsion-free, non-elementary hyperbolic group and $G$ be the group as in the exact sequence \eqref{ce}. If $Eu(p)\in H^2(\overline{G};\mathbb{Q})$ is nontrivial, then any homomorphism $\phi:G\to \Lambda$ factors through $p$; i.e.,  we have the diagram 
$
\xymatrix{G\ar[r]^{\phi}\ar[d]^p & \Lambda \\
\overline{G} \ar@{-->}[ru]&.
}
$
\end{lem}
\begin{proof}
Let $\alpha$ be a generator of the $\ZZ$ subgroup of $G$ as in \eqref{ce}. Since $\alpha$ is central in $G$, we know $\phi(G)$ should lie in the centralizer of $\phi(\alpha)$, which is a cyclic group. When $Eu(p)\neq 0\in H^2(\overline{G};\mathbb{Q})$, we know that $H^1(G;\mathbb{Q})=H^1(\overline{G};\mathbb{Q})$, which implies that any homomorphism $G\to \mathbb{Z}$ factors through $p$. Another way to see this is that the centralizer of some power of $\alpha$ is a product of commutators in $G$, which means that any homomorphism $G\to \ZZ$ factors through $p$.
\end{proof}
When $Q$ is a finite index subgroup of $G$, we obtain a similar result for $Q$.
\begin{cor}
Let $G$ satisfy the exact sequence \eqref{ce} and  $Q<G$ be a finite index subgroup. Let $\Lambda$ be either $\ZZ$ or a torsion-free, non-elementary hyperbolic group. If $Eu(p)\neq 0 \in H^2(\overline{G};\mathbb{Q})$,  
then any homomorphism $Q\to  \Lambda$ factors through $p$.
\end{cor}
\begin{proof}
 Let $Q$ be a finite index subgroup of $G$. Then we obtain a $\ZZ$-central extension
\begin{equation}\label{ce2}
1\to \ZZ\to Q\xrightarrow{p'=p|_Q} \overline{Q}\to 1.
\end{equation}
If $Eu(p)\neq 0\in H^2(\overline{G};\mathbb{Q})$, then $Eu(p')$ is nontrivial as well. This follows from the fact that the map $H^2(\overline{G};\QQ)\to H^2(\overline{Q};\QQ)$ is injective when $\overline{Q}$ is a finite index subgroup in $\overline{G}$. Then some nontrivial element in the $\mathbb{Z}$-subgroup of $G$ is also a commutator in $Q$. Then the corollary follows from Lemma \ref{centralextension}.
\end{proof}

The above method has been used in the following result of Putman and Bridson \cite[Theorem A]{Putman} \cite[Theorem A]{Bridson}, which we recall here.   Let $\Mod(S_g)$  be the the mapping class group of $S_g$; i.e., the group of connected components of the homeomorphism group of $S_g$. See \cite{FM} for an introduction on mapping class groups.
\begin{cor}[Putman, Bridson]
Let $g>2$ and let $\Gamma<\Mod(S_g)$ be a finite index subgroup. Let $\Lambda$ be either $\ZZ$ or a  torsion-free, non-elementary hyperbolic group. Then any homomorphism $\phi: \Gamma\to \Lambda$ satisfies  $\phi(T)=1$ for $T\in \Mod(S_g)$ a power of Dehn twist that is in $\Gamma$.
\end{cor}
\begin{proof}[Sketch of the proof]
The centralizer of a Dehn twist is a $\ZZ$-central extension of a short exact sequence where the Euler class is rationally nontrivial. See \cite{Putman} for more details. 
\end{proof}

Notice that Putman proved the result when the target is $\mathbb{Z}$ and Bridson proved it for actions on CAT(0)-space.  The above method does not work for $h\in \Mod(S_g)$ when $h^m$ is not a multi-twist (a product of powers of Dehn twist about disjoint curves) for some $m$. This is because if no power of $h$ is a multi-twist, a power of $h$ is a pseudo-Anosov element on a subsurface of $S_g$, which is never a product of commutators in its centralizer.

\subsection{Product of groups}
We now discuss homomorphisms to $\Lambda$ from a direct product of groups, which is an extension of \cite[Lemma 5.1]{Chen2}.
\begin{thm}\label{producthyperbolic}
Let $G_1,...,G_n$ be groups and let $\Gamma<G_1\times...\times G_n$ be a finite index subgroup. Let $\pi_i:\Gamma\to G_i$ be the $i$th projection map and let  $\Gamma_i$ be the image of $\pi_i$. 

\begin{enumerate}
\item The following decomposition holds:
\[H^1(\Gamma;\QQ)=\bigoplus_i \pi_i^*(H^1(\Gamma_i;\QQ)).
\]
\item
Let $\Lambda$ be a torsion-free, non-elementary hyperbolic group. Then any homomorphism $\phi:\Gamma\to \Lambda$ either factors through $\pi_i$ or its image is a cyclic group. 
\end{enumerate}
\end{thm}
\begin{proof}
We first assume that $\pi_i$ is surjective for all $i$, otherwise we replace the group $G_i$ with $\Gamma_i$. By induction, all we need is to prove the case when $n=2$. Let $K_1:=\Gamma\cap (G_1\times 1)$  and $K_2:\Gamma\cap (1\times G_2)$. We have the following exact sequences
\[
1\to K_1\to \Gamma\to G_2\to 1,
\]
\[
1\to K_1\to G_1\to G_1/K_1=G_2/K_2=G_1\times G_2/\Gamma\to 1.
\]
Now,
\[
H^1(K_1;\QQ)^{G_2}=H^1(K_1;\QQ)^{\Gamma}=H^1(K_1;\QQ)^{{\pi}_1(\Gamma)}=H^1(K_1;\QQ)^{G_1}=H^1(G_1;\QQ),
\]
which implies $H^1(\Gamma;\QQ)=\bigoplus_i \pi_i^*(H^1(\Gamma_i;\QQ))$.

For the second statement, let $\phi: \Gamma\to \Lambda$ be a homomorphism. Since $K_i$ is a finite index subgroup of $G_i$, it implies that $\Gamma$ contains $K:=K_1\times K_2$ as a finite index subgroup. To prove that $\phi$ factors through either $\pi_1$ or $\pi_2$, we only need to show that $\phi$ is trivial on either $K_1$ or $K_2$. If not, suppose that there exists $x_i\in K_i$ such that $\phi(x_i)\neq 1$. Since $x_1$ and $x_2$ commute with each other, their centralizers are the same cyclic group $C$. Since $K_2$ commutes with $x_1$ and $K_1$ commutes with $x_2$, we know that $\phi(K_1\times K_2)$ lies in $C$, which is a cyclic group. 

Now assume that $\phi(K)$ is a cyclic group generated by $a\in \Lambda$, but $\phi(\Gamma)$ is not cyclic. There is an element $\gamma\in \Gamma$ such that $\phi(\gamma)=b$ that is independent with $a$. Thus, no power of $a$ is in the group generated by $b$.  However, since $K$ is a finite index subgroup of $\Gamma$ and $\phi(K)$ is cyclic, we know that $\phi(\Gamma)$ is a finite extension of a cyclic group, which implies that a power of $a$ should lie in the group generated by $b$. This  contradicts our assumption on $a$ and $b$.
\end{proof}

\section{Surface braid groups}
In this section, we discuss first betti numbers of covers of surface braid groups and their homorphisms to torsion-free, non-elementary hyperbolic groups. 

\subsection{The case of the braid group of a closed surface}
\begin{proof}[Proof of Theorem \ref{bettinumber} and Theorem \ref{hyperbolic}]
Let 
\[e: \PConf_n(S_g)\to (S_g)^n\] be the natural embedding, where the image is the complement of the diagonal. We know that $PB_n(S_g)$ satisfies the following exact sequence by Lefschetz hyperplane theorem \[
1\to K\to PB_n(S_g)\xrightarrow{e_*} \pi_1(S_g)^n\to 1.
\]
Observe that $K$ is normally generated by $\{T_{ij}\}$, which geometrically is a loop around the diagonal 
\[
\triangle_{ij}:=\{(x_1,...,x_n)|x_i=x_j\}\subset S_g^n.\]
We will prove a stronger theorem that for any finite index subgroup $\Gamma<PB_n(S_g)$ satisfying $K<\Gamma$, any  homomorphism $\Gamma\to \Lambda$ factors through $e_*$ for $\Lambda$ a torsion-free, non-elementary hyperbolic group. Then Theorem \ref{bettinumber} and \ref{hyperbolic} follows from Theorem \ref{producthyperbolic},  \cite[Lemma 2.1]{Chen} and the fact that \[
e_{*}:H^1(\PConf_n(S_g);\mathbb{Q})\to H^1((S_g)^n;\mathbb{Q})\] is an isomorphism.
 
Let $\rho: \Gamma\to \Lambda$ be a homomorphism. The goal is to prove that $\rho(gT_{ij}g^{-1})=0$ for any $g\in PB_n(S_g)$ (notice that $g$ may not be in $\Gamma$, but $gT_{ij}g^{-1}\in K\subset \Gamma$). This proves that $\rho(K)=0$, which implies that $\Lambda$ factors through $e_*$.

Let us consider the centralizer of $T_{ij}$. The group $PB_n(S_g)$ can also be thought of as a point-pushing subgroup of the mapping class group $\Mod(S_{g,n})$, the connected component of the group of homeomorphisms of $S_g$ fixing $n$ marked points $m_1,...,m_n$ (see, e.g., \cite[Chapter 9]{FM} for more background on mapping class groups). Under this interpretation, the element $T_{ij}$ can also be thought as the Dehn twist around a simple closed curve surrounding points $m_i$ and $m_j$. The centralizer of $T_{ij}$ inside $PB_n(S_g)$ satisfies the following short exact sequence
\begin{equation}\label{1}
1\to \ZZ\xrightarrow{T_{ij}} C(T_{ij})\to PB_{n-1}(S_g)\to 1.
\end{equation}
Let $US_g$ be the unit circle bundle of the genus $g$ surface. The above short exact sequence is actually a pull-back of the following exact sequence
\begin{equation}\label{2}
1\to \ZZ\to \pi_1(US_g)\to \pi_1(S_g)\to 1
\end{equation}
by a forgetful map that forgets all points except the $i$th and $j$th points. To check whether the Euler class of \eqref{1} is trivial or not, we only need to compute the pull-back of the Euler class from \eqref{2}, which is a multiple of the fundamental class of $S_g$. By  \cite[Lemma 3.1]{Chen}, we see that the pull-back of the fundamental class by any forgetful map $p_i$ is not rationally trivial. This implies that the Euler class of \eqref{1} is nontrivial, which shows that $T_{ij}$ vanish under any homomorphism $C(T_{ij})\to\Lambda$ by Proposition \ref{centralextension}.  The same method applies to all conjugates of $T_{ij}$ as well. Since conjugates of all of the $T_{ij}$ generate $K$, we know that $\rho$ is trivial on $K$. 
\end{proof}

\subsection{The case of the braid group of a closed surface}
We now prove Theorem \ref{puncturehyperbolic}. Even though the idea is similar but the proof is more technical. 
\begin{proof}[Proof of Theorem \ref{puncturehyperbolic}]
For the same reason as the proof of Theorem \ref{bettinumber} and \ref{hyperbolic}, we have the following exact sequence
\[
1\to K\to PB_n(S_{g,p})\to \pi_1(S_{g,p})\times...\times \pi_1(S_{g,p})\to 1.
\]
We now consider $PB_n(S_{g,p})$ as a subgroup of $\Mod(S_{g,p+n})$. Let $\{m_1,...,m_n\}$ be the marked points and  $\{q_1,...,q_p\}$ be the punctures. Then $K$ is normally generated by conjugates of $T_{ij}$, which are Dehn twists about curves surrounding the points $m_i$ and $m_j$ for $1\le i,j\le n$.

The centralizer of $T_{ij}$ satisfies the following exact sequence:
\begin{equation}\label{3}
1\to \ZZ\xrightarrow{T_{ij}} C(T_{ij})\to PB_{n-1}(S_{g,p})\to 1.
\end{equation}
The only difference between punctured case and closed case is that  the Euler class of \eqref{3} is rationally trivial but the Euler class of \eqref{1} is rationally nontrivial. So we need a different strategy.






We now work with the case $n=2$ and the case $n=2$ implies the rest by induction. We will show in this case that $\rho(T_{ij})$ is trivial. Assume that  $\rho(T_{12})\neq 1$, where $T_{12}$ is the Dehn twist about a simple closed curve $c$ surrounding $m_1$ and $m_2$. We claim that there exists a simple closed curve $c$ surrounding  $m_1$ and $q_k$ or $m_2$ and $q_k$ for some $k$ such that $\rho(T_{c'})$ is nontrivial. Otherwise, $\rho$ is trivial over all of such simple closed curves; in which case we would know that $\rho$ factors through $PB_2(S_g)$, since Dehn twists about all of such simple closed curves generate the kernel of the natural homomorphism $PB_2(S_{g,p})\to PB_2(S_g)$. Then we conclude the theorem by Theorem \ref{hyperbolic}. Without loss of generality, we assume there exists $c'$ surrounding $m_1$ and $q_k$ such that $\rho(T_{c'})\neq 1$.

For convenience, we introduce the following notations. Let  $x_1=m_1$, $x_2=m_2$, $x_3=q_k$ and $x_4$ be the set of the rest of the punctures other than $q_k$, which are schematically represented by Figure \ref{haha} below on a plane $P\cong \mathbb{R}^2\subset S_{g,p}$. In Figure \ref{haha}, we call a simple closed curve $c_{i_1...i_j}$ if it surrounds a convex disk with punctures $x_{i_1},...,x_{i_j}$. We position the plane $P$ such that $c_{12}=c$ and $c_{13}=c'$. For example, Figure \ref{haha} has $c_{34},c_{4}, c_{13}$.
\begin{figure}[H]
\centering
 \includegraphics[width = 2in]{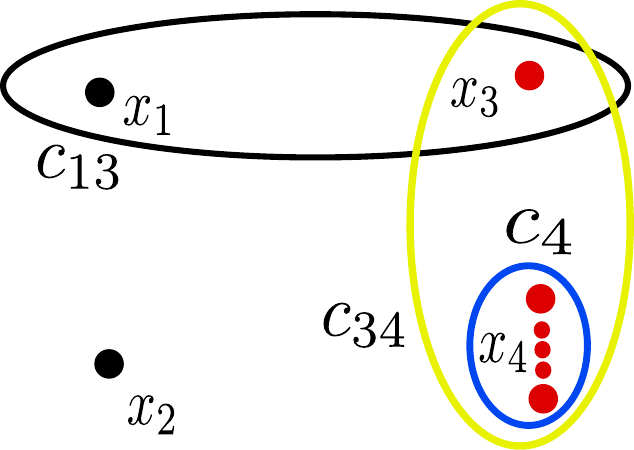}
 \caption{notation for curves}
 \label{haha}
\end{figure}

We continue to introduce more notations.
\begin{itemize}
\item Let $A_{12}$ be the Dehn twist about $c_{12}$.
\item Let $A_{i3}$ be the Dehn twist about $c_{i3}$ for $i\in \{1,2\}$.
\item Let $A_{i4}$ be the product of Dehn twist about $c_{i4}$ and a negative power of Dehn twist about $c_4$ for $i\in \{1,2\}$. 
\item Let $A_{123}$ be the Dehn twist about $c_{123}$ and let $A_{124}$ be the product of Dehn twist about  $c_{124}$ and  a negative power of Dehn twist about $c_4$. 
\item Let $A_{1234}$ be the product of Dehn twist about  $c_{1234}$ and  a negative power of Dehn twist about $c_{34}$. 
\end{itemize}
In defining the above list of elements, we sometimes multiply a negative power of Dehn twist about ${c_{4}}$ or ${c_{34}}$ so that all of $A_*$ lie in $PB_2(S_{g,p})$. On \cite[Page 97]{FM} and \cite[Page 119]{FM}, point-pushing map and disc-pushing map are defined as subgroups of $\Mod(S_{g,p+2})$. The point-pushing map of a loop is the mapping class in the isotopy class of pushing a marked point around a loop in $S_{g,p+1}$ (the other marked point stays still); the disc-pushing map is  of a loop is the mapping class in the isotopy class of pushing the disc with boundary $c_{12}$ around a loop in $S_{g,p}$. We have the following relations originate from the point(or disc)-pushing map or the lantern relation (see \cite[Chapter 5]{FM}) for more details about relations in mapping class groups).

\begin{enumerate}
\item $A_{12}A_{13}A_{23}=A_{123}$.
\item $A_{12}A_{14}A_{24}=A_{124}$.
\item $A_{123}A_{124}A_{12}^{-2}=A_{1234}A_{12}^{-1}$, which is the disc-pushing of $c_{12}$ around $c_{34}$. After forgetting $m_1$ and $m_2$, the curve $c_{34}$ is the boundary of a genus $g$ subsurface in $S_{g,p}$. Therefore,  $A_{1234}A_{12}^{-1}$ is a commutator in the centralizer of $A_{12}$.

\item $A_{12}A_{14}A_{13}=A_{1234}A_{234}^{-1}$, which is the point-pushing of $x_1$ around $c_{234}$. Similarly $A_{1234}A_{234}^{-1}$ is a commutator in the centralizer of $A_{234}$, $A_{23}$ and $A_{24}$ because the point-push of $x_1$ around other curves in the punctured surface do not intersect $c_{234}$.
\item $A_{12}A_{23}A_{24}=A_{1234}A_{134}^{-1}$, which is the  point-pushing of $x_2$ around $c_{134}$ and similarly is a commutator in the centralizer of $A_{134}$, $A_{13}$ and $A_{14}$.
\end{enumerate}

According to the assumption, we know that both $\rho(A_{12})$ and $\rho(A_{13})$ are nontrivial.  Therefore by (3) and the fact that the centralizer of $\rho(A_{12})$ is cyclic, we know that 
\[
\rho(A_{1234}A_{12}^{-1})=\rho(A_{123}A_{124}A_{12}^{-2})=1\]
By (5) that $A_{1234}A_{134}^{-1}$ is a commutator which  commutes with $A_{13}$ and  that $\rho(A_{13})\neq 1$, we obtain \[
\rho(A_{12}A_{23}A_{24})=\rho(A_{1234}A_{134}^{-1})=1.\] This implies that either $\rho(A_{23})\neq 1$ or $\rho(A_{24})\neq 1$, which further implies that the image under $\rho$ of the commutator in (4) is trivial, as it lies in the centralizers of $\rho(A_{23})$ and $\rho(A_{24})$.

Therefore we know that $\rho(A_{12}A_{14}A_{13})=1$ in (4). Multiplying (1),(2) and $\rho(A_{123}A_{124})=\rho(A_{12}^{2})$ gives us 
\[
\rho(A_{13}A_{14}A_{23}A_{24})=1.
\]
It contradicts the result of the multiplication of (4) and (5) under $\rho$  \[
\rho(A_{13}A_{14}A_{23}A_{24}A_{12}^2)=1. \qedhere\]
\end{proof}

\section{The pure braid group of the disc}
In the section, we will prove Theorem \ref{Pn}. We first introduce a generating set for $PB_n$. Recall that $PB_n$ is the pure mapping class group of the disk with $n$-marked points; i.e., $\pi_0(\Diff(\mathbb{D}_n))$, where $\Diff(\mathbb{D}_n)$ is the group of diffeomorphisms of $\mathbb{D}$ fixing $n$ marked points pointwise . Consider the disk with $n$ marked points $\mathbb{D}_n$ in Figure \ref{p1}. 

\begin{figure}[H]
\minipage{0.30\textwidth}
  \includegraphics[width=\linewidth]{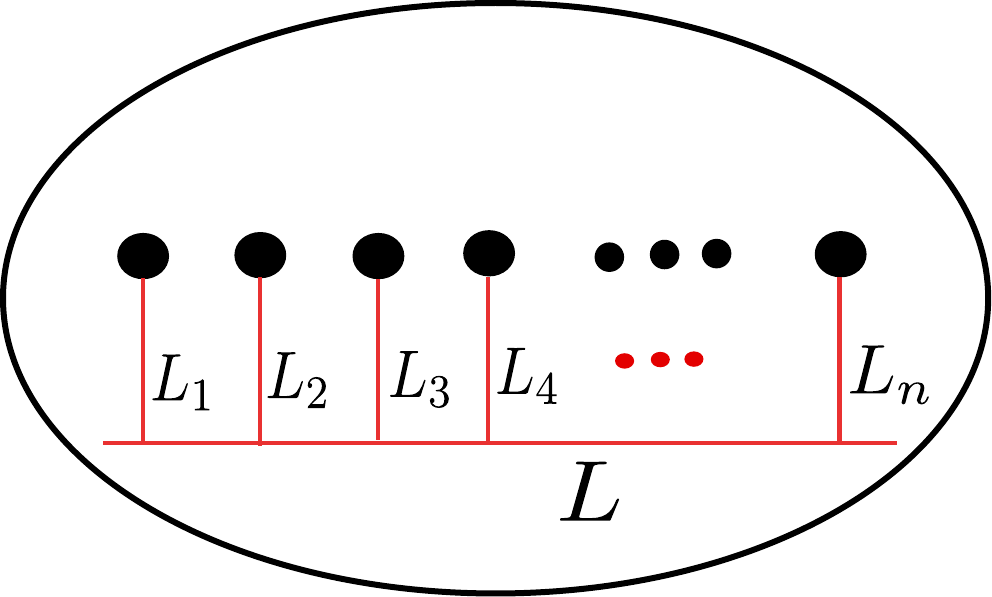}
  \caption{$\mathbb{D}_{n}$.}\label{p1}
\endminipage\hfill
\minipage{0.30\textwidth}
  \includegraphics[width=\linewidth]{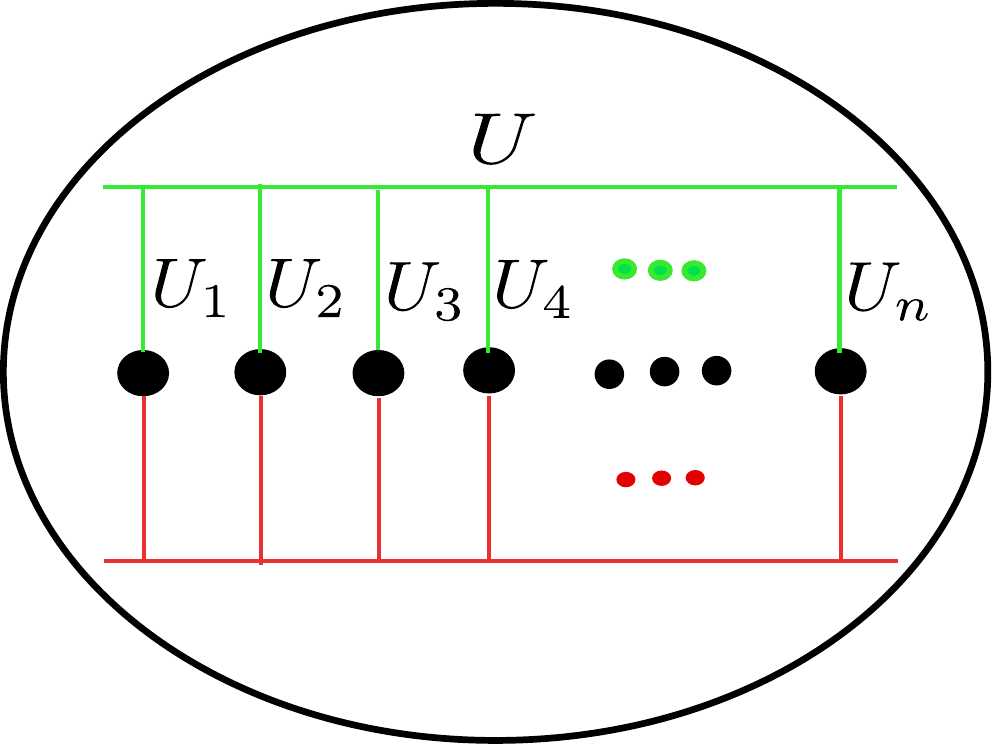}
  \caption{}\label{p3}
\endminipage\hfill
\minipage{0.30\textwidth}
  \includegraphics[width=\linewidth]{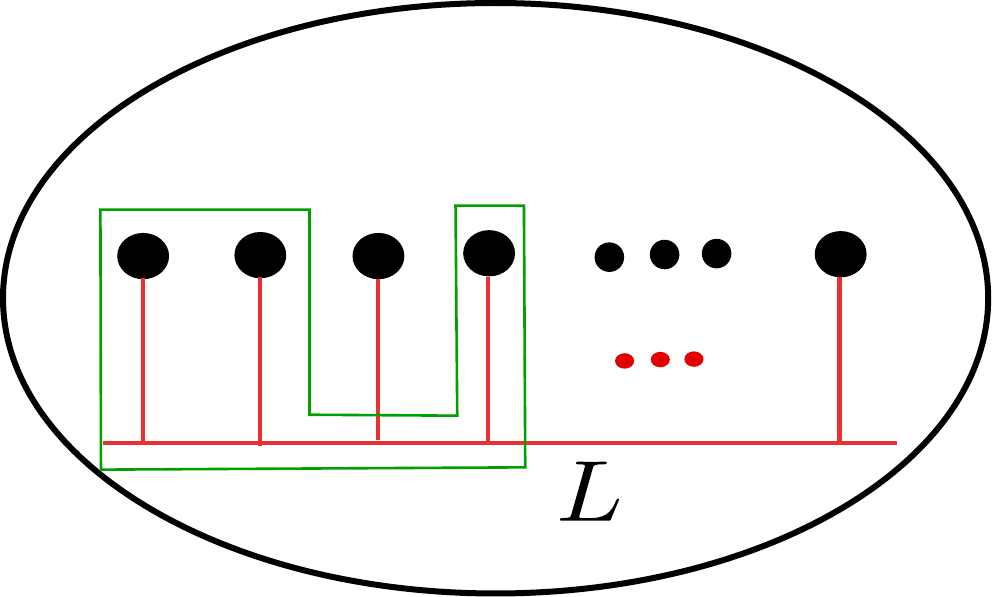}
  \caption{$a_{124}$.}
  \label{p2}
\endminipage\hfill
\end{figure}

Let $L$ be a line segment below all the marked points $x_1,...,x_n$. Let $L_1, ..., L_n$ be line segments connecting $x_1, ..., x_n$ to $L$ as in Figure \ref{p1}. Similarly, let $U$ be a line segment above all marked points and let $U_1,...,U_n$ be line segments connecting $x_1, ..., x_n$ to $U$ as shown in Figure \ref{p3}.

For $\{i_1, ..., i_k\}\subset \{1, ..., n\}$, let $a_{i_1i_2...i_k}$ (resp. $a_{i_1i_2...i_k}'$) be the boundary curve of the tubular neighborhood of 
$\bigcup_{m=1}^k L_{i_m}\cup L$ (resp. $\bigcup_{m=1}^k U_{i_m}\cup U$). Let $T_{i_1i_2...i_k}$ (resp. $T_{i_1i_2...i_k}'$) be the Dehn twist about $a_{i_1i_2...i_k}$ (resp. $a_{i_1i_2...i_k}'$). Figure \ref{p2} gives  an example of a curve representing $a_{124}$. The following proposition about generating sets of $PB_n$ is classical and can be found in \cite[Theorem 2.3]{MM}.
\begin{prop}
Both $\{T_{ij} |1\le i<j\le n\}$ and $\{T_{ij}' |1\le i<j\le n\}$ are generating sets for $PB_n$.
\end{prop}

Before proving Theorem \ref{Pn}, we analyze the map $R_*: PB_4\to PB_3$. 
\begin{fact}\label{RQ}
The map $R_*$ satisfies the following relations
\[
R_*(T_{12})=T_{12}, R_*(T_{23})=T_{23}, R_*(T_{34})=T_{12}, 
\]\[
R_*(T_{13})=T_{13},  R_*(T_{24})=T_{12}^{-1}T_{23}^{-1}T_{12}T_{13}T_{23}\text{ and }
R_*(T_{14})=T_{23}.\]
\end{fact}
\begin{proof}
The ``resolving the quartic" map is also a map $R':\Conf_4(\mathbb{C})\to \Conf_3(\mathbb{C})$, where $\Conf_n(\mathbb{C})$ is the space of unordered n-tuples of points in $\mathbb{C}$. It induces a homomorphism on the fundamental groups $R_*': B_4\to B_3$, where $B_n:=\pi_1(\Conf_n(\mathbb{C}))$ is the braid group. 

The above relations can be computed using the map on the braid group $R_*: B_4\to B_3$. Let $\sigma_i$ be the standard generating set for the braid group $B_n$. The map $R_*'$ satisfies that 
\[R_*'(\sigma_1)=\sigma_1, R_*'(\sigma_2)=\sigma_2 \text{ and }R_*'(\sigma_3)=\sigma_1.\] The restriction of $R_*'$ on $PB_n$ can be computed from it.
\end{proof}
The homomorphism $RQ$ satisfies the following.
\begin{fact}\label{RQQ}
Let $\{a,b\}$ be the natural generating set of $F_2$. Then the homomorphism $RQ$ satisfies 
\[
RQ(T_{12})=a, RQ(T_{23})=b, RQ(T_{34})=a,\]
\[RQ(T_{13})=b^{-1}a^{-1}, RQ(T_{24})=a^{-1}b^{-1}\text{ and }RQ(T_{14})=b.\]
\end{fact}
\begin{proof}
The quotient map $Q: PB_3\to F_2$ satisfies that $Q(T_{12})=a$ and $Q(T_{23})=b$ and $Q(T_{13})=b^{-1}a^{-1}$. Thus by Fact \ref{RQ}, we conclude the proof.
\end{proof}

We now start the proof of Theorem \ref{Pn}.
\begin{proof}[Proof of Theorem \ref{Pn}]
Let $\Lambda$ be a torsion-free, non-elementary hyperbolic group and 
let $\rho: PB_n\to \Lambda$ be a homomorphism such that the image is not a cyclic group. Let $Z_n$ be the generator of the center of $PB_n$, which is the Dehn twist about the boundary curve. We prove this theorem by induction on $n$. Firstly, when $n=3$, if $\rho(Z_3)\neq 1$, then $\rho(PB_3)$ lies in the centralizer of $\rho(Z_3)$, which is a cyclic group. Thus, we know that $\rho$ factors through $Q$. We assume now that the theorem is true for $n-1$. 

Since $\{T_{ij} |1\le i<j\le n\}$ is a generating set of $PB_n$ and the image of $\rho$ is not cyclic, we know there exist two elements $\rho(T_{i,j})$ and $\rho(T_{i',j'})$ that do not commute. It implies that $\{i,j\}\cap \{i',j'\}\neq \emptyset$ because otherwise $T_{i,j}$ and $T_{i',j'}$ commute. Observe that  there exists an element $g\in B_n$ such that $gT_{i,j}g^{-1}=T_{12}$ and $gT_{i',j'}g^{-1}=T_{23}$. Then up to a conjugation by $g$, we assume that $a:=\rho(T_{12})$ and $b:=\rho(T_{23})$ do not commute (a conjugation by $g$ is equivalent to a rename of punctures).

We split the rest of the proof into two cases depending on whether $\rho(T_{34})$ is trivial or not.  

\para{The case when \boldmath$\rho(T_{34})=1$}
By the lantern relation, we have
\begin{equation}\label{lantern1}
T_{123}T_{34}T_{124}=T_{12}T_{1234}.
\end{equation}
Since $\rho(T_{1234})$ commutes with both $a=\rho(T_{12})$ and $b=\rho(T_{23})$, we know that $\rho(T_{1234})=1$. Similarly, we know that $\rho(T_{123})=1$. Thus we have $\rho(T_{124})=a$. Since $\rho(T_{14})$ commutes with both $b=\rho(T_{23})$ and $a=\rho(T_{124})$, we know that $\rho(T_{14})=1$. Then by the lantern relation, we have 
\begin{equation}\label{lantern2}
T_{12}T_{24}T_{14}=T_{124}.
\end{equation}
Thus we have $\rho(T_{24})=1$. Since $\rho(T_{4j})$ commutes with both $a=\rho(T_{12})$ and $b=\rho(T_{23})$ for any $j>4$ , we know that $\rho(T_{4j})=1$ for $j>4$. Observe that $\{T_{4j}|1\le j \neq 4\le n\}$ is a generating set of the kernel of the the forgetful map that forgets the fourth points $F_4: PB_n\to PB_{n-1}$. Then we know that $\rho$ factors through $F_4$, which by induction induces the result.

\para{The case when \boldmath$\rho(T_{34})\neq 1$} 
We first prove the claim that $\rho(T_{34})=a$. If not, by equation \eqref{lantern1}, we know that $\rho(T_{124})\neq 1$ and it commutes with $a$. Since $\rho(T_{14})$ commutes with two independent elements $\rho(T_{124})$ and $\rho(T_{23})$, we know that $\rho(T_{14})=1$. Thus by equation \eqref{lantern2}, we know that $\rho(T_{24})\neq 1$ and commutes with $a$. The element $\rho(T_{234})$ is trivial since it commutes with two independent elements $\rho(T_{23})$ and $\rho(T_{34})$. This contradicts with the  lantern relation $T_{23}T_{34}T_{24}=T_{234}$ because $\rho(T_{34}T_{24})$ commutes with $a$ but $\rho(T_{23})$ does not.

We now prove that \[
\rho(T_{14})=b, \rho(T_{13})=b^{-1}a^{-1}\text{ and  }\rho(T_{24})=a^{-1}b^{-1}.\] Since $\rho(T_{123})$ commutes with $\rho(T_{12})$ and $\rho(T_{23})$, we know that $\rho(T_{123})=1$. Similarly, we know that $\rho(T_{234})=1$. By the lantern relation $T_{12}T_{23}T_{13}=T_{123}$, we know that $\rho(T_{13})=b^{-1}a^{-1}$. Then by the lantern relation $T_{23}T_{34}T_{24}=T_{234}$, we know that $\rho(T_{24})=a^{-1}b^{-1}$. Since $\rho(T_{124})$ commutes with $a=\rho(T_{12})$ and $a^{-1}b^{-1}=\rho(T_{24})$, we know that $\rho(T_{124})=1$. Now by the lantern relation $T_{12}T_{24}T_{14}=T_{124}$ we obtain that $T_{14}=b$.  When $n=4$, by Fact \ref{RQQ}, we know that $\rho$ factors through $RQ$.

When $n>4$, we claim that $\rho(T_{5j}')=1$ for $1\le j\neq 5\le n$, which would imply that $\rho$ factors through the forgetful map $F_5: PB_n\to PB_{n-1}$, since $\{T_{5j}'|1 \le j\neq 5\le n\}$ is a generating set of the kernel of $F_5$. The equation $\rho(T_{51}')=1$ follows from the commutativity of $\rho(T_{5j}')$ with $\rho(T_{23})$ and $\rho(T_{34})$. The triviality of $\rho(T_{5j}')$ for other $j$ follows from the same reason, and the result follows.

\end{proof}





\bibliography{citing}{}
\vskip 0.3cm
\noindent
Department of Mathematics\\
California Institute of Technology\\
Pasadena, CA 91125,  USA \\
chenlei@caltech.edu
\end{document}